\documentclass[a4paper,reqno]{amsart}

\usepackage{enumerate}
\usepackage{amssymb} % For \gtrsim and \lesssim
\usepackage[colorlinks=true]{hyperref}
\hypersetup{urlcolor=blue, citecolor=cyan}
\hypersetup{citecolor=blue}

\numberwithin{equation}{section}

\newtheorem{thm}{Theorem}[section]
\newtheorem{lem}[thm]{Lemma}

\newtheorem{prop}[thm]{Proposition}
\theoremstyle{definition}
\newtheorem{rem}[thm]{Remark}

\newcommand\Step[1]{\par\medskip\noindent {\sc Step~#1.}\quad}

\newcommand\R{{\mathbb R}}
\newcommand\C{{\mathbb C}}

\newcommand\Tma{T_{\mathrm{max}}}

\newcommand\Eqdef{\stackrel{\text{\tiny def}}{=}}

\newcommand\Loc{{\mathrm{loc}}}

\newcommand\goto{\mathop{\longrightarrow}}

\newcommand\MScN[1]{\href{http://www.ams.org/mathscinet-getitem?mr=#1}{\nolinkurl{(#1)}}}
\newcommand\DOI[1]{\href{http://dx.doi.org/#1}{(doi: \nolinkurl{#1})}}
\newcommand\LINK[1]{\href{#1}{(link: \nolinkurl{#1})}}

\newcommand\DIb{v_0 }

\begin{document}

\title{Finite-time blowup for a Schr\"o\-din\-ger equation with nonlinear source term}

\def\shorttitle{Finite-time blowup}

\author[T. Cazenave]{Thierry Cazenave$^1$}
\email{\href{mailto:thierry.cazenave@sorbonne-universite.fr}{thierry.cazenave@sorbonne-universite.fr}}

\author[Y. Martel]{Yvan Martel$^{2}$}
\email{\href{mailto:yvan.martel@polytechnique.edu}{yvan.martel@polytechnique.edu}}

\author[L. Zhao]{Lifeng Zhao$^3$}
\email{\href{mailto:zhaolf@ustc.edu.cn}{zhaolf@ustc.edu.cn}}

\address{$^1$Sorbonne Universit\'e \& CNRS, Laboratoire Jacques-Louis Lions,
B.C. 187, 4 place Jussieu, 75252 Paris Cedex 05, France}

\address{$^2$CMLS, \'Ecole Polytechnique, CNRS, 91128 Palaiseau Cedex, France}

\address{$^3$Wu Wen-Tsun Key Laboratory of Mathematics and School of Mathematical Sciences, University of Science and Technology of China, Hefei 230026, Anhui, China}



\thanks{The third author thanks the hospitality of Professor Frank Merle when he visited IHES and Professor Yvan Martel when he visited CMLS, \'Ecole Polytechnique, where part of the work was done.}

\begin{abstract}
We consider the nonlinear Schr\"o\-din\-ger equation
\[ 
u_t = i \Delta u + | u |^\alpha u \quad \mbox{on $\R^N $, $\alpha>0$,}
\]
for $H^1$-subcritical or critical nonlinearities: $(N-2) \alpha \le 4$.
Under the additional technical assumptions $\alpha\geq 2$ (and thus $N\leq 4$), we construct
$H^1$ solutions that blow up in  finite time with explicit blow-up profiles and blow-up rates. In particular, blowup can occur at any given finite set of points of $\R^N$.

The construction involves explicit functions $U$, solutions of the ordinary differential equation $U_t=|U|^\alpha U$. In the simplest case, $U(t,x)=(|x|^k-\alpha t)^{-\frac 1\alpha}$ for $t<0$, $x\in \R^N$. For $k$ sufficiently large, $U$ satisfies $|\Delta U|\ll U_t$ close to the blow-up point $(t,x)=(0,0)$, so that it is a suitable approximate solution of the problem. To construct an actual solution $u$ close to $U$, we use energy estimates and a compactness argument.
\end{abstract}

\maketitle

%\begin{center} 
%\today
%\end{center} 

\section{Introduction}

We consider the nonlinear Schr\"o\-din\-ger equation
\begin{equation} \label{NLS1} 
u_t = i \Delta u + | u |^\alpha u
\end{equation} 
on $\R^N $, where $\alpha >0$ satisfies
\begin{equation} \label{fDFG0} 
(N-2) \alpha \le 4. 
\end{equation}
Under assumption~\eqref{fDFG0}, equation~\eqref{NLS1} is $H^1$-subcritical or critical, so that the corresponding Cauchy problem is locally well posed in $H^1 (\R^N )$ -- see e.g. \cite{CLN10}.

Equation~\eqref{NLS1} is a member of the more general family of complex Ginzburg-Landau equations 
\[u_t = \zeta \Delta u + \xi | u |^\alpha u \quad \mbox{ where $\zeta , \xi \in \C$ and $ \Re \zeta \ge 0 $.}\] 
It is proved in~\cite[Theorem~1.1]{CazenaveCDW-Fuj} that (for $ \alpha <2/N$) equation~\eqref{NLS1} has no global in time $H^1$ solution that remains bounded in $H^1$. In other words, every $H^1$ solution blows up, in finite or infinite time.
The question of finite-time blowup is left open in~\cite{CazenaveCDW-Fuj}. 
It seems that no standard argument based on obstruction to global existence (Levine's method, variance argument) is applicable. 

The purpose of this article is to construct solutions of~\eqref{NLS1} that blow up in finite time. 
For technical reasons we require
\begin{equation} \label{fCT1} 
\alpha \ge 2 .
\end{equation} 
(The condition $\alpha \ge 1$ is used in the proof of estimates~\eqref{fPE1}-\eqref{fPE2}, the stronger condition $\alpha \ge 2$ is used in formula~\eqref{fGCA9}.)
Conditions~\eqref{fDFG0}-\eqref{fCT1} impose $N \le 4$. 
More precisely, the allowed range of powers is
\begin{equation*} 
\begin{cases} 
\alpha \in [2, \infty )\quad & N=1, 2 \\
\alpha \in [2, 4] \quad & N=3 \\
 \alpha =2\quad & N=4
\end{cases} 
\end{equation*} 

Our first blow-up result, related to single point blowup with a simple asymptotic profile is the following. 

\begin{thm} \label{eThm2} 
Let $N\ge 1$ and let $\alpha >0$ satisfy~\eqref{fDFG0} and~\eqref{fCT1}. Given $A>0$ and $k> \max\{10, \frac {N\alpha } {2} \}$, let $U$ be defined by
\begin{equation} \label{fCT8:2}
U(t, x) = (-t) ^{- \frac {1} \alpha } f ( (-t)^{ - \frac {1} {k}} x)
\quad t<0, \, x\in \R^N ,
\end{equation} 
where
\begin{equation} \label{fCT8:3}
 f (x)= ( \alpha + A |x|^k)^{- \frac {1} {\alpha }} .
\end{equation} 
It follows that there exist a solution $u\in C((-\infty , 0), H^1 (\R^N ) ) $ of~\eqref{NLS1} and $ \mu >0$ such that
\begin{equation} \label{fEEN26:2} 
 \| u(t) - U(t) \| _{ H^1 } \lesssim ( -t ) ^{ \mu } 
\end{equation} 
as $t\uparrow 0$. 
In particular, $u$ blows up at $t=0$ and
\begin{gather} 
(-t) ^{ \frac {1} {\alpha } - \frac {N} {2 k } } \| u(t) \| _{ L^2 } \goto _{ t\uparrow 0 } \| f \| _{ L^2 } \label{fEEN27:2} \\
 (-t) ^{ \frac {1} {\alpha } - \frac {N-2} {2 k } } \| \nabla u(t) \| _{ L^2 } \goto _{ t\uparrow 0 } \| \nabla f \| _{ L^2 }. \label{fEEN29:2} 
\end{gather} 
Moreover
\begin{equation} \label{fEEN30:2} 
u(t, x) \goto _{ t \uparrow 0 } A^{-\frac {1} {\alpha }} |x|^{-\frac {k} {\alpha }}
\end{equation} 
in $H^1 (\{ |x|>\varepsilon \} ) $ for every $\varepsilon >0$.
\end{thm} 

Theorem~\ref{eThm2} is a particular case of a more general result (Theorem~\ref{eThm1} below), where asymptotic profiles more general than~\eqref{fEEN30:2} are allowed. 
Before stating precisely this more general result, we need to make precise our assumptions on the asymptotic profile. 
We consider an integer $J\ge 1$, real numbers $\rho , \nu $, $(k_j) _{ 1\le j\le J } $, $(\eta _{ j, \beta }) _{ 1\le j\le J }$ where $\beta $ is a multi-index with $ |\beta |\le 3$, and points $(x_j) _{ 1\le j\le J }\subset \R^N $ such that 
\begin{equation} \label{fHYP1} 
\begin{cases} 
\rho , \eta_{j,\beta}, K >0, \quad \nu > \frac {N\alpha } {2} \\
 | x_j - x_\ell |\ge 2\rho , \quad j\not = \ell \\
 k_1 \le \cdots \le k_J \\
 k_1 > 10, \quad k_J > \frac {N\alpha } {2} 
\end{cases} 
\end{equation} 
Let $\phi \in C^3 ( \R^N , \R )$ satisfy
\begin{equation} \label{fHYP2} 
\begin{cases} 
\phi (x) >0 \quad \quad x\not = x_j, 1\le j\le J, \\
\displaystyle | y |^{- k_j + |\beta |} | D^\beta \phi (x_j +y) | \goto _{ |y| \to 0 } \eta _{ j, \beta } \quad \quad 1\le j\le J , 0\le |\beta |\le 3, \\
\displaystyle \liminf _{ |x|\to \infty } |x|^{-\nu } \phi (x) >0 \\
\displaystyle \limsup _{ |x| \to \infty } | x| ^{ - \nu } | D^\beta \phi (x) | < \infty \quad \quad 1\le |\beta |\le 3 .
\end{cases} 
\end{equation} 
Our main result is the following

\begin{thm} \label{eThm1} 
Let $N\ge 1$ and let $\alpha >0$ satisfy~\eqref{fDFG0} and~\eqref{fCT1}. Let $\phi \in C^3 ( \R^N , \R )$ satisfy~\eqref{fHYP1}-\eqref{fHYP2}. 
Let $U$ be defined by
\begin{equation} \label{fCT8}
U(t, x) = (- \alpha t + \phi (x))^{- \frac {1} {\alpha }}\quad t<0, \, x\in \R^N.
\end{equation} 
It follows that there exist a solution $u\in C((-\infty , 0), H^1 (\R^N ) ) $ of~\eqref{NLS1} and $ \mu >0$ such that
\begin{equation} \label{fEEN26} 
 \| u(t) - U(t) \| _{ H^1 } \lesssim ( -t ) ^{ \mu } 
\end{equation} 
as $t\uparrow 0$. 
In particular, $u$ blows up at $t=0$ and
\begin{gather} 
 (-t) ^{ \frac {1} {\alpha } - \frac {N} { 2 k_J } } \| u(t) \| _{ L^2 } \goto _{ t\uparrow 0 } a \label{fEEN27} \\
 (-t) ^{ \frac {1} {\alpha } +\theta } \| \nabla u(t) \| _{ L^2 } \goto _{ t\uparrow 0 } b \label{fEEN29} 
\end{gather} 
where $a, b>0$ and
\begin{equation} \label{fEU1:2:b2} 
\theta = 
\begin{cases} 
 \frac {2 - N} {2 k_J } & N=1,2 \\
 \frac {2 - N} {2 k_1 } & N=3,4
\end{cases} 
\end{equation} 
Moreover
\begin{equation} \label{fEEN30} 
u(t, x) \goto _{ t \uparrow 0 } \phi (x)^{-\frac {1} {\alpha }}
\end{equation} 
in $H^1 (\omega ) $ for every $\displaystyle \omega \subset \subset \R^N \setminus \mathop{\cup} _{ 1\le j\le J } \{ x_j\}$.
\end{thm} 

\begin{rem} \label{eRM1} 
Here are some comments on Theorems~\ref{eThm2} and~\ref{eThm1}.
\begin{enumerate}[{\rm (i)}] 

\item \label{eRM1:1}
Theorem~\ref{eThm2} presents the simplest result with the choice of only two parameters $A$ and $k$.
The parameter $k$ has to be taken sufficiently large so that the ansatz $U$ satisfies $|\Delta U|\ll U_t$
(a similar strategy is used in \cite{CollotGM}, see the comments below). Note that the choice of $k$ determines the blow-up rates in \eqref{fEEN27:2}--\eqref{fEEN29:2}.
More parameters can chosen in Theorem~\ref{eThm1}, which allows arbitrary locations for the blow-up points and flexibility on the blow-up rates. It is easy to construct explicitly functions $\phi$ satisfying \eqref{fHYP2}.
\item \label{eRM1:2}
It follows from~\eqref{fEEN27} and~\eqref{fEEN29} that both $\| u(t) \| _{ L^2 }$ and $\| \nabla u(t) \| _{ L^2 }$ blow up as $t\uparrow 0$.

\item \label{eRM1:3}
The asymptotic profile $f = \phi ^{-\frac {1} {\alpha }} $ of $u$ as $t\uparrow 0$, given by~\eqref{fEEN30}, has the following properties: $f>0$, $ |x|^{ \frac {\nu} {\alpha }} f(x) $ is bounded as $ |x| \to \infty $ and $f$ is $C^3$ except at the points $x_j$, where $f$ has a singularity like $ |x-x_j|^{- \frac {k_j} {\alpha }}$. 

\item \label{eRM1:4} The solutions constructed in Theorems~\ref{eThm2} and~\ref{eThm1} are global for $t<0$.
Actually, this is a general fact : we show in Proposition~\ref{eRem3:1} that equation~\eqref{NLS1} is globally well-posed
in $H^1$ in the negative sense of time.

\end{enumerate} 
\end{rem} 

We prove Theorem~\ref{eThm1} by using the strategy of~\cite{Merle1}.
More precisely, we consider the sequence $(u_n) _{ n\ge 1 }$ of solutions of~\eqref{NLS1} defined by $u_n (- \frac {1} {n}) = U (- \frac {1} {n})$, where $U$ is defined by~\eqref{fCT8}. 
It follows that $u_n$ is defined on $(-\infty , -\frac {1} {n}]$. 
Since $U_t = | U |^\alpha U$, and $\Delta U$ is small compared to $U_t$, $U$ is almost a solution of~\eqref{NLS1}. Following the idea of~\cite{Martel} (see also~\cite{RaphaelS} in the blow-up context) we estimate the solutions $u_n$ by energy arguments. 
Note that the nice behavior of equation \eqref{NLS1} backwards in time, already discussed in Remark~\ref{eRM1} \eqref{eRM1:4}, is important in this step.
Finally, passing to the limit as $n\to \infty $ yields the solution $u$ of Theorem~\ref{eThm1}. 

The solution $u$ given by Theorem~\ref{eThm1} blows up at $t=0$ like the function $U$ defined by~\eqref{fCT8}. Since $U_t = |U|^\alpha U$, we see that the solution $u$ displays an ODE-type blowup. 

We recall that ODE-type blowup has been   intensively studied for several other nonlinear equations.
For the nonlinear heat equation, the type of blowup obtained in Theorems~\ref{eThm2} and~\ref{eThm1} is called \emph{flat} blowup in the literature and was recently investigated independently of our work in
~\cite{CollotGM} (see also previous references there) with applications to the Burgers equation. Even if the blow-up profile is identical (see~\S 4.1 of \cite{CollotGM}), the strategy to proceed with the construction of an actual solution of the equation is different in this paper.

Apart from such unstable forms of blowup, ODE-type blowup was also much studied as a stable form of blowup. We  refer to \cite{CollotGM, MerleZ, NoZa} and to references there for results in the parabolic context. For the semilinear wave
and quasilinear wave equations, we refer \emph{e.g.} to \cite{Alinhac,MerleZ2,MerleZ3,Speck} and to the references there.

The rest of the paper is organized as follows. In Section~\ref{sTBE}, we study the Cauchy problem for the backwards version of equation~\eqref{NLS1}, and in Section~\ref{sEST1}, we derive various estimates on the function $U$. Section~\ref{sEST2} is devoted to the construction and estimates of the approximate blow-up solutions, and the proof of Theorem~\ref{eThm1} is completed in Section~\ref{sCOMP} by passing to the limit in the approximate solutions. 

\section{The backwards equation} \label{sTBE} 

In this section, we prove that the Cauchy problem for the backwards equation obtained by changing $t$ to $-t$ in~\eqref{NLS1} is globally well-posed in $H^1 (\R^N ) $, under assumption~\eqref{fDFG0}. 

\begin{prop} \label{eRem3:1} 
Let $N\ge 1$ and assume~\eqref{fDFG0}. 
Given any $\DIb \in H^1 (\R^N ) $, there exists a solution $u\in C([0,\infty ), H^1 (\R^N ) )\cap C^1 ( [0, \infty ), H^{-1} (\R^N ) )$ of 
\begin{equation} \label{fBCK1} 
\begin{cases} 
 v_t = -i \Delta v - |v|^\alpha v \\
v(0) = \DIb .
\end{cases} 
\end{equation}
In addition, if $\tau >0$ and $u_1, u_2 \in L^\infty ((0,\tau ), H^1 (\R^N ) )\cap W^{1, \infty } ((0,\tau ), H^{-1} (\R^N ) )$ are two solutions of~\eqref{fBCK1} on $(0, \tau )$, then $u_1 = u_2$.
\end{prop} 

\begin{proof} 
Assumption~\eqref{fDFG0} ensures that equation~\eqref{fBCK1} is $H^1$-subcritical or critical, so that local well-posedness in $H^1 (\R^N ) $ follows from standard arguments. (See e.g.~\cite{CLN10}.)

We now prove that the Cauchy problem~\eqref{fBCK1} is globally well posed in $H^1 (\R ^N) $. 
Indeed, let $\DIb\in H^1 (\R^N ) $, and let $\Tma$ be the maximal existence time of the corresponding solution $v$ of~\eqref{fBCK1}. 
Multiplying the equation by $ \overline{v} $ and taking the real part yields
\begin{equation*} 
\frac {d} {dt} \| v (t) \| _{ L^2 }^2 = -2 \| v (t) \| _{ L^{\alpha +2} }^{\alpha +2} \le 0
\end{equation*} 
so that $ \| v(t)\| _{ L^2 } \le \| v(0) \| _{ L^2 }$ for all $t\ge 0$. Next, we multiply the equation by $-\Delta \overline{v} $ and take the real part. We obtain (see~\eqref{fDFG2}-\eqref{fDFG3} below)
\begin{equation} \label{fCC1} 
\begin{split} 
\frac {d} {dt} \| \nabla v (t) \| _{ L^2 }^2 & = -2 \Re \int _{ \R^N } \nabla \overline{v} \cdot \nabla ( |v|^\alpha v) 
\\ & =- \frac {\alpha +2} {2} \int _{ \R^N } |v|^\alpha |\nabla v|^2 - \frac {\alpha } {2} \Re \int _{ \R^N } |v|^{\alpha -2}v^2 (\nabla \overline{v})^2 \\ & \le - \int _{ \R^N } |v|^\alpha |\nabla v|^2 \le 0 
\end{split} 
\end{equation} 
so that $ \| \nabla v (t)\| _{ L^2 } \le \| \nabla v (0) \| _{ L^2 }$ for all $0\le t< \Tma$. These formal calculations can be justified by standard arguments, see e.g.~\cite{Ozawa}. 
This proves global existence in the subcritical case $(N-2) \alpha < 4$. In the critical case $N\ge 3$ and $\alpha =\frac {4} {N-2}$, we use the right-hand side of~\eqref{fCC1} to obtain
\begin{equation*} 
\int _0^{\Tma } \int _{ \R^N } |v|^\alpha |\nabla v|^2 \le \| \nabla v (0) \| _{ L^2 } ^2.
\end{equation*} 
Since
\begin{equation*} 
 |v|^\alpha |\nabla v|^2 \ge |v|^\alpha |\nabla |v|\, |^2= \frac {4} {(\alpha +2)^2} | \nabla |v|^{ \frac {\alpha +2} {2} } |^2
\end{equation*} 
we deduce by Sobolev's inequality that
\begin{equation*} 
\int _0^{\Tma } \| v(t) \| _{ L^{ \frac {N(\alpha +2)} {N-2}} }^{\alpha +2} < \infty .
\end{equation*} 
We define $2< r < N$ by
\begin{equation} \label{fCC3} 
\frac {N(\alpha +2)} {N-2} = \frac {Nr} {N- r }
\end{equation} 
so that
\begin{equation} \label{fCC4} 
\int _0^{\Tma } \| v(t) \| _{ L^{ \frac {Nr } {N-r}} }^{\alpha +2} < \infty .
\end{equation}
Since by~\eqref{fCC3} 
\begin{equation*} 
\frac {2} {\alpha +2 } = N \Bigl( \frac {1} {2} - \frac {1} { r } \Bigr)
\end{equation*} 
we deduce from~\eqref{fCC4} that $\Tma =\infty $. (See~\cite[Remark~4.5.4~(ii)]{CLN10}.)

Finally, we prove the stronger uniqueness property, so we consider $\tau >0$ and two solutions $u_1, u_2 \in L^\infty ((0,\tau ), H^1 (\R^N ) )\cap W^{1, \infty } ((0,\tau ), H^{-1} (\R^N ) )$ of~\eqref{fBCK1} on $(0, \tau )$.
Setting 
\begin{equation*} 
w (t) = u_1 (t) - u_2 (t) , \quad 0\le t\le \tau 
\end{equation*} 
 we see that $w\in L^\infty ((0,\tau ) , H^1 (\R^N ) ) \cap W^{1, \infty } ( (0, \tau ), H^{-1} (\R^N ) )$ satisfies
\begin{equation*} 
w_t = - i \Delta w - | u_1 |^\alpha u_1 + | u_2 |^\alpha u_2
\end{equation*} 
in $L^\infty ((0,\tau ), H^{-1} (\R^N ) ) $. 
Moreover, the map $t\mapsto \| w(t) \| _{ L^2 }^2$ is in $W^{1, \infty } (0,\tau )$ and
\begin{equation} \label{fUNQ3:b1:b} 
\frac {1} {2}\frac {d} {dt} \| w(t) \| _{ L^2 }^2 = \langle w_t, w\rangle _{ H^{-1}, H^1 }
\end{equation} 
for a.a. $t\in (0,\tau )$. 
Taking the $H^{-1}-H^1$ duality product of equation~\eqref{fBCK1} with $w$ and applying~\eqref{fUNQ3:b1:b}, we deduce that for a.a. $t\in (0,\tau )$
\begin{equation} \label{fUNQ4:b} 
\frac {1} {2}\frac {d} {dt} \| w(t) \| _{ L^2 }^2 = - \langle | u_1 |^\alpha u_1 - | u_2 |^\alpha u_2 , w\rangle _{ H^{-1}, H^1 }
= - \Re \int ( | u_1 |^\alpha u_1 - | u_2 |^\alpha u_2 ) \overline{w} .
\end{equation} 
We recall that
\begin{equation} \label{fEE1} 
 \Re [ ( | z_1 |^\alpha z_1 - | z_2 |^\alpha z_2) ( \overline{z_1} - \overline{z_2} ) ] \ge 0
\end{equation} 
for all $z_1,z_2 \in \C$. 
Indeed,
\begin{equation*} 
\begin{split} 
 \Re ( | z_1 |^\alpha z_1 - | z_2 |^\alpha z_2 ) ( \overline{z_1} - \overline{z_2} ) & = |z_1|^{\alpha +2} + |z_2|^{\alpha +2} - ( |z_1|^\alpha + |z_2|^\alpha ) \Re z_1 \overline{z_2} \\ & \ge |z_1|^{\alpha +2} + |z_2|^{\alpha +2} - ( |z_1|^\alpha + |z_2|^\alpha ) |z_1| \, |z_2| \\ & = ( |z_1|^{\alpha +1} - |z_2|^{\alpha +1} ) ( |z_1| - |z_2| ) \ge 0.
\end{split} 
\end{equation*} 
Since $w (0) =0$, we deduce from~\eqref{fUNQ4:b} and~\eqref{fEE1} that $w\equiv 0$ on $(0,\tau )$.
\end{proof} 

\section{Estimates of $U$} \label{sEST1} 

In this section, we establish various estimates on the function $U$ defined by~\eqref{fCT8}.

\begin{lem} \label{eEU1} 
Let $\phi \in C^3 ( \R^N , \R )$ satisfy~\eqref{fHYP1}-\eqref{fHYP2}.
If $U$ is given by~\eqref{fCT8}, then $U\in C((-\infty , 0), H^3 (\R^N ) )$ and
\begin{align} 
 \| U (t) \| _{ L^\infty } & \lesssim (-t)^{-\frac {1} {\alpha } } , \label{fEU1:1} \\ 
 \| \nabla U (t) \| _{ L^\infty } & \lesssim (-t)^{-\frac {1} {\alpha }- \frac {1} { k_1 }} , \label{fEU1:2}\\
 \| \Delta U (t) \| _{ L^2 } & \lesssim (-t)^{-\frac {1} {\alpha }- \frac {4 - N} {2 k_1 }} , \label{fEU1:3} \\
 \| \nabla \Delta U (t) \| _{ L^2 } & \lesssim (-t)^{-\frac {1} {\alpha }- \frac { 6 - N} {2 k_1 }} , \label{fEU1:4}
\end{align} 
as $t\uparrow 0$, and
\begin{align} 
 (-t)^{ \frac {1} {\alpha }- \frac { N} {2 k_J }} \| U (t) \| _{ L^2 } & \goto _{ t\uparrow 0 } a>0 , \label{fEU1:1:b1} \\
 (-t)^{ \frac {1} {\alpha } + \theta } \| \nabla U (t) \| _{ L^2 } & \goto _{ t\uparrow 0 } b>0 , \label{fEU1:2:b1} 
\end{align} 
where $\theta $ is given by~\eqref{fEU1:2:b2}. 
\end{lem} 

\begin{proof} We proceed in three steps.

\Step1 Proof of estimates~\eqref{fEU1:1}-\eqref{fEU1:2}.\quad 
We have $\| U (t) \| _{ L^\infty } = (- \alpha t )^{- \frac {1} {\alpha }}$, which implies~\eqref{fEU1:1}. 
Moreover, 
\begin{equation} \label{fEU1:5}
\nabla U= - \frac {1} {\alpha } U^{\alpha +1} \nabla \phi 
\end{equation} 
We note that by~\eqref{fCT8} and~\eqref{fHYP2}, there exists $R>0$ such that
\begin{equation} \label{fEU1:5b1}
U \le \phi (x) ^{- \frac { 1} {\alpha }}\le C | x| ^{- \frac {\nu } {\alpha }} 
\end{equation} 
 for $ |x| \ge R$. Applying~\eqref{fHYP2}, we deduce that $ | \nabla U|$ is bounded independently of $t<0$ for $ |x| \ge R$. 
Given $r>0 $, let 
\begin{equation*} 
E_r= \{x\in \R^N ;\, |x| < R \text{ and } |x-x_j| > r \text{ for all } 1\le j\le J \}. 
\end{equation*} 
It follows that $\sup \{ U; \, x\in E_\rho , t<0 \} <\infty $, so that $ | \nabla U|$ is bounded independently of $t < 0$ and $x\in E_\rho $.
For $ x\in B( x_j, \rho )$, assumption~\eqref{fHYP2} and formula~\eqref{fEU1:5} imply
\begin{equation*} 
 |\nabla U |\le \frac {1} {\alpha } U U^\alpha | \nabla \phi | \le C (-t) ^{- \frac {1} {\alpha }} \frac { |x-x_j|^{k_j -1}} { -t + |x-x_j|^{k_j }} \le C (-t) ^{- \frac {1} {\alpha } - \frac {1} {k_j }}
\end{equation*} 
 which proves~\eqref{fEU1:2}. 
 
 \Step2 Proof of estimates~\eqref{fEU1:3}-\eqref{fEU1:4}.\quad 
 We have
 \begin{equation} \label{fEU1:6}
\Delta U= - \frac {1} {\alpha } U^{\alpha +1} \Delta \phi + \frac {\alpha +1} {\alpha ^2}U^{ 2 \alpha +1} | \nabla \phi |^2
\end{equation} 
and
\begin{equation}
\label{fEU1:7} 
\begin{split} 
\nabla \Delta U = & - \frac {1} {\alpha } U^{\alpha +1} \nabla \Delta \phi + \frac {\alpha +1} {\alpha ^2} U^{ 2 \alpha +1} [ \Delta \phi \nabla \phi + \nabla ( | \nabla \phi |^2 )] \\ & - \frac { (\alpha +1) (2\alpha +1) } {\alpha ^3} U^{ 3 \alpha +1} | \nabla \phi |^2 \nabla \phi 
\end{split} 
\end{equation} 
We observe that by~\eqref{fEU1:6}, \eqref{fEU1:7}, and~\eqref{fHYP2}, there exists $R>0$ such that
\begin{equation*} 
 |\Delta U |+ | \nabla \Delta U| \le C |x| ^{- \frac {\nu } {\alpha }} 
\end{equation*} 
for all $ |x| \ge R$ and $t < 0$, so that 
\begin{equation*} 
\sup _{ t<0 } \| \Delta U \| _{ L^2 (\{ |x|>R \}) } + \| \nabla \Delta U \| _{ L^2 (\{ |x|>R \}) } <\infty .
\end{equation*} 
Moreover, $U$ is bounded on $E$, so that by~\eqref{fEU1:6} and~\eqref{fEU1:7},
\begin{equation*} 
\sup _{ t < 0 } \| \Delta U \| _{ L^2 ( E ) } + \| \nabla \Delta U \| _{ L^2 ( E ) } <\infty .
\end{equation*} 
On $ B( x_j, \rho )$, we deduce from~\eqref{fHYP2} and formulas~\eqref{fEU1:6} and~\eqref{fEU1:7} that
\begin{equation*} 
 |\Delta U |+ |x-x_j| \, | \nabla \Delta U| \le C ( - t + |x - x_j |^{k_j}) ^{ -1 - \frac {1} {\alpha }} |x-x_j|^{k_j -2}
\end{equation*} 
and estimates~\eqref{fEU1:3}-\eqref{fEU1:4} easily follow.
 
 \Step3 Proof of~\eqref{fEU1:1:b1}-\eqref{fEU1:2:b1}.\quad 
 It follows from~\eqref{fEU1:5b1}, \eqref{fEU1:5}, \eqref{fHYP1} and~\eqref{fHYP2} that $ \| U \| _{ L^2 (\{ |x|>R \}) }$ and $ \| \nabla U \| _{ L^2 (\{ |x|>R \}) }$ are bounded independently of $t < 0$. Furthermore, it is clear that $ \| U \| _{ L^2 (E _r ) }$ and $ \| \nabla U \| _{ L^2 (E _r) }$ are also bounded independently of $t <0$, for every $r>0$.
Therefore, we need only calculate $ \| U \| _{ L^2 ( B(x_j, r) ) }$ and $ \| \nabla U \| _{ L^2 ( B(x_j, r) ) }$ for $r>0$ small. 
 Assumption~\eqref{fHYP2} implies that $\phi ( x_j +y) \sim \eta _{ j, 0 } | y |^{k_j}$ for $ |y|$ small, so that for small $r>0 $
\begin{equation*} 
\int _{ B( x_j, r ) } U^2 \sim \int _{ B( 0, r ) } ( - \alpha t + \eta _{ j, 0 } |y |^{k_j})^{-\frac {2} {\alpha }}
\sim (-t) ^{- \frac {2} {\alpha } + \frac {N} {k_j}} \int _{ \R^N } (\alpha + \eta _{ j, 0 } |y|^{k_j}) ^{-\frac {2} {\alpha }}.
\end{equation*} 
The limit~\eqref{fEU1:1:b1} easily follows. 
To prove~\eqref{fEU1:2:b1}, we observe that by~\eqref{fEU1:5} and~\eqref{fHYP2}
\begin{equation*} 
\int _{ B( x_j, r ) } |\partial _\ell U|^2 \sim \int _{ B( 0, r ) } \eta _{ j, \ell } ^2 |y|^{ 2 k_j - 2} ( - \alpha t + \eta _{ j, 0 } |y |^{k_j})^{-\frac {2 (\alpha +1) } {\alpha }}
\end{equation*} 
and we conclude as above. 

\Step4 $U\in C(( - \infty , 0), H^3 (\R^N ) )$.\quad 
Given $ \tau <0$, we deduce from~\eqref{fHYP2}, \eqref{fEU1:5}, \eqref{fEU1:6} and~\eqref{fEU1:7} that
\begin{equation*} 
 |U |+ | \nabla U| + | \Delta U| + | \nabla \Delta U| \le C (1 + |x| )^{- \frac {\nu } {\alpha }} 
\end{equation*} 
for all $x\in \R^N $ and $ t\le \tau $. Since $ \frac {\nu } {\alpha } > \frac {N} {2}$, the conclusion easily follows by dominated convergence.
\end{proof} 

\section{Construction and estimates of the approximate solutions} \label{sEST2} 

We observe that by~\eqref{fCT8}, $U$ satisfies
\begin{equation*} 
U_t = |U|^\alpha U .
\end{equation*} 
We now construct approximate solutions that behave like $U$. 
More precisely, we set
\begin{equation} \label{fDFTN} 
T_n= - \frac {1} {n}
\end{equation} 
for $n\ge 1$, and we consider the solution $u_n$ of equation~\eqref{NLS1} with the initial condition
\begin{equation} \label{fGCA0} 
u_n ( T_n)= U( T_n) \in H^1 (\R^N ) 
\end{equation} 
It follows from Proposition~\ref{eRem3:1} that $u_n$ is well defined, $u_n \in C((-\infty , T_n], H^1 (\R^N ) )$.
We now establish estimates that are uniform in $n\ge 1$.
For this we set
\begin{equation} \label{fGCA1} 
u_n = U + \varepsilon _n
\end{equation} 
so that $\varepsilon _n \in C((-\infty , T_n], H^1 (\R^N ) )$.

\begin{lem} \label{eME1} 
If $\varepsilon _n$ is as above, then there exist $C, \delta ,\mu >0$ such that
\begin{equation} \label{feME1:1} 
 \| \varepsilon _n (t) \| _{ H^1 } \le C (T_n- t) ^{ \mu }
\end{equation} 
for all $T_n - \delta \le t \le T_n$.
\end{lem} 

\begin{proof} 
In the calculations that follow, we drop the index $n$. 
Moreover, we let $k= k_1$, where $k_1$ is given by~\eqref{fHYP1}. 
In addition, we make formal calculations, which can be justified for instance by the method of~\cite{Ozawa}.
It is convenient to set 
\begin{equation} \label{fDFG1} 
g (z )= |z|^\alpha z
\end{equation} 
for all $z\in \C$. 
We note that
\begin{align} 
\partial _z g (z) &= \frac {\alpha +2} {2} |z|^\alpha \label{fDFG2} \\
\partial _{ \overline{z} } g (z) &= \frac {\alpha } {2} |z|^{\alpha -2} z^2 \label{fDFG3}
\end{align} 
for all $z\in \C$. 
We will also use the estimates
\begin{gather} 
 | \partial _{ {z} } g( u+v ) - \partial _z g(u) - \partial _z g( v ) | \lesssim ( |u|^{\alpha -1} |v | + |v |^{\alpha -1} |u|) 
 \label{fPE1} \\
 | \partial _{ \overline{z} } g( u+v ) - \partial _{ \overline{z} } g(u) - \partial _{ \overline{z} } g( v ) | \lesssim ( |u|^{\alpha -1} |v | + |v |^{\alpha -1} |u|) \label{fPE2} 
\end{gather} 
for all $u,v\in \C$.
We establish~\eqref{fPE2}, the proof of~\eqref{fPE1} being similar. 
To prove~\eqref{fPE2}, we consider the three cases: 
$ |u| \ge 2 |v| ; \frac {1} {2} |v| \le |u| \le 2 |v| ; |u|\le \frac {1} {2} |v|$. 
The second case is immediate, because then $ |u|$ and $ |v|$ are equivalent.
Next, the first and third cases are equivalent, because the expressions on the right-hand side of~\eqref{fPE2} are symmetric in $u,v$. Therefore, we consider only the first case and, assuming without loss of generality $u\not = 0$, we have
\begin{equation*} 
\frac {2} {\alpha } [ \partial _{ \overline{z} } g( u+v ) - \partial _{ \overline{z} } g(u) - \partial _{ \overline{z} } g( v ) ] 
= | u+v |^{\alpha -2} (u+ v)^2 - |u|^{\alpha -2} u^2 - |v|^{\alpha -2} v^2 .
\end{equation*} 
Since $ | \, |v|^{\alpha -2} v^2 |= |v|^{\alpha }\le |u|^{\alpha -1} |v|$ (recall that $\alpha \ge 1$), we need only estimate $ | | u+v |^{\alpha -2} (u+ v)^2 - |u|^{\alpha -2} u^2 | $. We have
\begin{equation*} 
 | | u+v |^{\alpha -2} (u+ v)^2 - |u|^{\alpha -2} u^2 | = |u|^\alpha | | h(z) | 
\end{equation*} 
where
\begin{equation*} 
h(z) = | 1+ z |^{\alpha -2} (1 + z)^2 - 1,\quad z= \frac {v} {u} .
\end{equation*} 
The function $h $ is $C^1$ on $\{ z\in \R^2;\, |z|\le \frac {1} {2} \}$ and $h(0)= 0$, so that there exists a constant $C$ such that $ |h(z)| \le C |z|$ for $ |z|\le \frac {1} {2}$; and so
\begin{equation*} 
 | | u+v |^{\alpha -2} (u+ v)^2 - |u|^{\alpha -2} u^2 | \le C |u|^\alpha \frac { |v|} { |u|}= C |u|^{\alpha -1} |v|
\end{equation*} 
which proves the desired estimate. 

The equation for $\varepsilon = \varepsilon _n$ is, with the notation~\eqref{fDFG1}
\begin{equation} \label{eEFP1} 
\begin{cases} 
\varepsilon_t = i \Delta \varepsilon + [g (U+\varepsilon )- g (U) ] + i \Delta U \\
\varepsilon ( T_n ) =0 .
\end{cases} 
\end{equation} 
Multiplying by $ \overline{\varepsilon } $ and taking the real part, we obtain after integration by parts
\begin{equation*} 
\frac {1} {2} \frac {d} {dt} \| \varepsilon \| _{ L^2 }^2 = \Re \int [ g (U+\varepsilon )- g(U) ] \overline{\varepsilon } + \Re \int i (\Delta U ) \overline{\varepsilon } 
\end{equation*} 
By~\eqref{fEE1}, the first term on the right-hand side is nonnegative, so that
\begin{equation*} 
 \frac {d} {dt} \| \varepsilon \| _{ L^2 } \ge - \| \Delta U \| _{ L^2 } 
\end{equation*} 
hence, using~\eqref{fEU1:3}, 
\begin{equation} \label{fGCA2} 
 \frac {d} {dt} \| \varepsilon \| _{ L^2 } \ge - C (-t)^{-\frac {1} {\alpha }- \frac {4-N} {2 k} } 
\end{equation} 
for $-1 \le t<0$.
We observe that by~\eqref{fCT1} and~\eqref{fHYP1} 
\begin{equation*} 
0< \frac {1} {\alpha } + \frac {4-N} {2k} \le \frac {1} {2 }+ \frac {3 } {2k} < 1 
\end{equation*} 
so that integrating~\eqref{fGCA2} on $(t, T_n)$ yields
\begin{equation} \label{fELD1} 
 \| \varepsilon \| _{ L^2 } \le C (T_n -t)^{ 1 -\frac {1} {\alpha }- \frac {4 - N} {2 k} }
 \le C (T_n -t)^{ 1 -\frac {1} {\alpha }- \frac {3} {2 k} }, \quad -1 \le t\le T_n.
\end{equation} 
We now multiply~\eqref{eEFP1} by $-\Delta \overline{\varepsilon } $, take the real part and integrate by parts. Since 
\begin{align*} 
\nabla g (U+ \varepsilon ) & = \partial _z g (U+\varepsilon ) (\nabla U + \nabla \varepsilon ) + \partial _{ \overline{z} } g (U+\varepsilon ) (\nabla U + \nabla \overline{\varepsilon} ) \\
\nabla g(U) & = \partial _z g (U ) \nabla U + \partial _{ \overline{z} } g (U ) \nabla U 
\end{align*} 
we obtain
\begin{equation} \label{fGCA3} 
\begin{split} 
\frac {1} {2} \frac {d} {dt} \| \nabla \varepsilon \| _{ L^2 }^2 & = 
\Re \int \partial _z g( U+\varepsilon ) (\nabla U\cdot \nabla \overline{\varepsilon } + |\nabla \varepsilon |^2 )
\\ 
& + \Re \int \partial _{ \overline{z} } g( U+\varepsilon ) (\nabla U\cdot \nabla \overline{\varepsilon } + (\nabla \overline{\varepsilon} )^2 ) 
\\
& - \Re \int \partial _z g( U ) \nabla U\cdot \nabla \overline{\varepsilon }
- \Re \int \partial _{ \overline{z} } g( U ) \nabla U\cdot \nabla \overline{\varepsilon } 
\\ 
& - \Re \int i \nabla \Delta U \cdot \nabla \overline{\varepsilon }
\end{split} 
\end{equation} 
It follows from~\eqref{fDFG2}-\eqref{fDFG3} that
\begin{equation*} 
\begin{split} 
\Re [ \partial _z g( U+\varepsilon ) & |\nabla \varepsilon |^2 + \partial _{ \overline{z} } g( U+\varepsilon ) (\nabla \overline{\varepsilon} )^2 ] \\ &= \frac {\alpha +2} {2} |U+\varepsilon |^\alpha | \nabla \varepsilon |^2
 + \frac {\alpha } {2} |U+\varepsilon |^{\alpha -2} \Re [ (U+\varepsilon )^2 (\nabla \overline{\varepsilon } )^2] 
\\ & \ge |U+\varepsilon |^\alpha | \nabla \varepsilon |^2
\end{split} 
\end{equation*} 
so that~\eqref{fGCA3} yields
\begin{equation} \label{fGCA4} 
\begin{split} 
\frac {1} {2} \frac {d} {dt} \| \nabla \varepsilon \| _{ L^2 }^2 & \ge 
 \int |U+\varepsilon |^\alpha | \nabla \varepsilon |^2
\\ 
& + \Re \int [ \partial _{ {z} } g( U+\varepsilon ) - \partial _z g(U) - \partial _z g( \varepsilon ) ] \nabla U\cdot \nabla \overline{\varepsilon } 
\\
& + \Re \int [ \partial _{ \overline{z} } g( U +\varepsilon ) - \partial _{ \overline{z} } g( U ) - \partial _{ \overline{z} } g( \varepsilon ) ] \nabla U\cdot \nabla \overline{\varepsilon } 
\\ 
& + \Re \int [ \partial _{ {z} } g( \varepsilon ) - \partial _{ \overline{z} } g( \varepsilon ) ] \nabla U\cdot \nabla \overline{\varepsilon } - \Re \int i \nabla \Delta U \cdot \nabla \overline{\varepsilon } \\
& \Eqdef I_1 + I_2 + I_3 + I_4 + I_5 
\end{split} 
\end{equation} 
Applying~\eqref{fEU1:4}, we have
\begin{equation} \label{fGCA5} 
 | I_5 | \le \| \nabla \Delta U \| _{ L^2 } \| \nabla \varepsilon \| _{ L^2 }
 \lesssim ( -t)^{-\frac {1} {\alpha }- \frac {5} {2k} } \| \nabla \varepsilon \| _{ L^2 } . 
\end{equation} 
In view of~\eqref{fPE1}-\eqref{fPE2}, we obtain 
\begin{equation} \label{fGCA6:b1} 
\begin{split} 
 |I_2 + I_3| & \lesssim \int [ U^{\alpha -1} |\varepsilon | + |\varepsilon |^{\alpha -1} U] |\nabla U|\, |\nabla \varepsilon | \\ & \lesssim [ \|\varepsilon \| _{ L^2 } \|U \| _{ L^\infty }^{\alpha -1}+ \|\varepsilon \| _{ L^{2\alpha -2} }^{\alpha -1} \| U\| _{ L^\infty } ] \|\nabla \varepsilon \| _{ L^2 } \| \nabla U \| _{ L^\infty } 
\end{split} 
\end{equation} 
By Gagliardo-Nirenberg's inequality
\begin{equation} \label{fGCA10:b1} 
 \| \varepsilon \| _{ L^{2\alpha -2 }}^{\alpha -1} \lesssim \| \nabla \varepsilon \| _{ L^2 }^{\frac {N} {2} (\alpha -2)} \| \varepsilon \| _{ L^2 }^{ \frac {2(N -1)- \alpha (N-2)} {2} },
\end{equation} 
so that~\eqref{fGCA6:b1}, \eqref{fEU1:1}, \eqref{fEU1:2} and~\eqref{fELD1} yield
\begin{equation} \label{fGCA6} 
\begin{split} 
 |I_2 + I_3| \lesssim & \|U \| _{ L^\infty }^{\alpha -1} \| \nabla U \| _{ L^\infty } \|\varepsilon \| _{ L^2 } \|\nabla \varepsilon \| _{ L^2 } 
 \\ & + \| U\| _{ L^\infty } \| \nabla U \| _{ L^\infty } \| \varepsilon \| _{ L^2 }^{ \frac {2(N -1)- \alpha (N-2)} {2} } \| \nabla \varepsilon \| _{ L^2 }^{ 1 + \frac {N} {2} (\alpha -2)} 
 \\ \lesssim &
 (T_n - t ) ^{- \frac {1} {\alpha } - \frac {5} { 2 k }} \|\nabla \varepsilon \| _{ L^2 } 
 + (T_n - t ) ^{ \sigma _1 } \| \nabla \varepsilon \| _{ L^2 }^{ 1 + \frac {N} {2} (\alpha -2)} 
\end{split} 
\end{equation} 
where
\begin{equation} \label{fDSI1} 
\sigma _1 = - \frac {2} {\alpha } - \frac {1} {k} + \Bigl( 1 -\frac {1} {\alpha }- \frac {3} {2 k} \Bigr) \Bigl( \frac {2(N -1)- \alpha (N-2)} {2} \Bigr) .
\end{equation} 
Next,
\begin{equation} \label{fGCA7} 
\begin{split} 
 |I_4| & \lesssim \int |\varepsilon |^\alpha |\nabla \varepsilon | \, |\nabla U| \lesssim
 \int |\varepsilon | \, |\varepsilon |^{\alpha - 1} |\nabla \varepsilon | \, |\nabla U| \\
 & \lesssim \int |\varepsilon | [ | U + \varepsilon |^{\alpha -1} + U^{\alpha -1} ] |\nabla \varepsilon | \, |\nabla U| \\
 & \lesssim \|\varepsilon \| _{ L^2 } \|U \| _{ L^\infty }^{\alpha -1} \|\nabla \varepsilon \| _{ L^2 } \| \nabla U \| _{ L^\infty } + \int |\varepsilon |\, | U + \varepsilon |^{\alpha -1} |\nabla \varepsilon | \, |\nabla U| 
\end{split} 
\end{equation} 
The first term in the right-hand side of~\eqref{fGCA7} appears in~\eqref{fGCA6:b1} and is estimated by the right-hand side of~\eqref{fGCA6}, so we estimate the last term in~\eqref{fGCA7}. 
By Cauchy-Schwarz
\begin{equation*} 
 \int |\varepsilon |\, | U + \varepsilon |^{\alpha -1} |\nabla \varepsilon | \, |\nabla U| \le \Bigl( \int |U+\varepsilon |^\alpha |\nabla \varepsilon |^2 \Bigr)^{\frac {1} {2}} \Bigl( \int |\varepsilon |^2 |U+\varepsilon |^{\alpha -2} |\nabla U |^2 \Bigr)^{\frac {1} {2}}
\end{equation*} 
so that 
\begin{equation} \label{fGCA8} 
 \int |\varepsilon |\, | U + \varepsilon |^{\alpha -1} |\nabla \varepsilon | \, |\nabla U| \le \delta \int |U+\varepsilon |^\alpha |\nabla \varepsilon |^2+ \frac {1} {4\delta }\int |\varepsilon |^2 |U+\varepsilon |^{\alpha -2} |\nabla U |^2 
\end{equation} 
for every $\delta >0$.
Since $\alpha \ge 2$, we have $ |U+\varepsilon |^{\alpha -2} \lesssim |U |^{\alpha -2}+ |\varepsilon |^{\alpha -2} $, hence
\begin{equation} \label{fGCA9} 
\begin{split} 
\int |\varepsilon |^2 |U+\varepsilon |^{\alpha -2} |\nabla U |^2 & \lesssim \int |\varepsilon |^2 |U |^{\alpha -2} |\nabla U |^2 + \int |\varepsilon |^\alpha |\nabla U |^2 \\ & \lesssim \| \varepsilon \| _{ L^2 }^2 \| U\| _{ L^\infty }^{\alpha -2} \| \nabla U\| _{ L^\infty }^2 + \| \varepsilon \| _{ L^\alpha }^\alpha \| \nabla U\| _{ L^\infty }^2 .
\end{split} 
\end{equation} 
By Gagliardo-Nirenberg's inequality
\begin{equation} \label{fGCA10} 
 \| \varepsilon \| _{ L^\alpha }^\alpha \lesssim \| \nabla \varepsilon \| _{ L^2 }^{\frac {N} {2} (\alpha -2)} \| \varepsilon \| _{ L^2 }^{ \frac {2N- \alpha (N-2)} {2} }
\end{equation} 
thus we deduce from~\eqref{fGCA9}, \eqref{fEU1:1}, \eqref{fEU1:2} and~\eqref{fELD1} that
\begin{equation} \label{fGCA9:b1} 
\int |\varepsilon |^2 |U+\varepsilon |^{\alpha -2} |\nabla U |^2 \lesssim 
(T_n -t)^{ 1- \frac {2} {\alpha } - \frac {5} {k } } + 
(T_n -t)^{ \sigma _2 }
 \| \nabla \varepsilon \| _{ L^2 }^{\frac {N} {2} (\alpha -2)} 
\end{equation} 
where
\begin{equation} \label{fDFS2} 
\sigma _2 = -\frac {2} {\alpha } - \frac {2} {k} + \Bigl( 1 -\frac {1} {\alpha }- \frac {3} {2 k} \Bigr) \Bigl( \frac {2N- \alpha (N-2)} {2} \Bigr) 
\end{equation} 
We deduce from~\eqref{fGCA6}, \eqref{fGCA7}, \eqref{fGCA8} and~\eqref{fGCA9:b1} that 
\begin{equation*} 
\begin{split} 
 |I_2 + I_3 + I_4| & - \int |U+\varepsilon |^\alpha |\nabla \varepsilon |^2 
\\ \lesssim (T_n - t ) ^{- \frac {1} {\alpha } - \frac {5} { 2 k }} & \|\nabla \varepsilon \| _{ L^2 } 
 + (T_n - t ) ^{ \sigma _1 } \| \nabla \varepsilon \| _{ L^2 }^{ 1 + \frac {N} {2} (\alpha -2)} 
 \\& + (T_n -t)^{ 1- \frac {2} {\alpha } - \frac {5} {k } } + 
(T_n -t)^{ \sigma _2 }
 \| \nabla \varepsilon \| _{ L^2 }^{\frac {N} {2} (\alpha -2)} 
\end{split} 
\end{equation*} 
hence~\eqref{fGCA4} and~\eqref{fGCA5} yield
\begin{equation} \label{fGCA13} 
\begin{split} 
- \frac {d} {dt} \| \nabla \varepsilon \| _{ L^2 }^2 \lesssim &
(T_n - t ) ^{- \frac {1} {\alpha } - \frac {5} { 2 k }} \|\nabla \varepsilon \| _{ L^2 } 
 + (T_n - t ) ^{ \sigma _1 } \| \nabla \varepsilon \| _{ L^2 }^{ 1 + \frac {N} {2} (\alpha -2)} 
 \\& + (T_n -t)^{ 1- \frac {2} {\alpha } - \frac {5} {k } } + 
(T_n -t)^{ \sigma _2 }
 \| \nabla \varepsilon \| _{ L^2 }^{\frac {N} {2} (\alpha -2)} 
\end{split} 
\end{equation} 
We note that~\eqref{fCT1} and~\eqref{fHYP1} imply
\begin{equation*} 
1 - \frac {1} {\alpha } - \frac {3} {2k} \ge \frac {1} {3} 
\end{equation*} 
and one deduces easily that
\begin{equation*} 
\rho \Eqdef 1+ \min \{ \sigma _1, \sigma _2 \} > 0 .
\end{equation*} 
Moreover,
\begin{equation*} 
- \frac {1} {\alpha } - \frac {5} { 2 k } \ge - \frac {3} {4}, \quad 1- \frac {2} {\alpha } - \frac {5} {k }\ge - \frac {1} {2}
\end{equation*} 
and it follows from~\eqref{fGCA13} that
\begin{equation} \label{fGCA14} 
- \frac {d} {dt} \| \nabla \varepsilon \| _{ L^2 }^2 \lesssim 
 (T_n - t ) ^{ -1 + \widetilde{\rho} } \Bigl( 1 + \| \nabla \varepsilon \| _{ L^2 }^{ 1 + \frac {N} {2} (\alpha -2)} \Bigr)
\end{equation} 
where 
\begin{equation*} 
 \widetilde{\rho } = \min \Bigl\{ \frac {1} {4}, \rho \Bigr\}. 
\end{equation*} 
We now set
\begin{equation} \label{fGCA15} 
\tau _n = \inf \{ t\in [ - 1, T_n]; \, \| \nabla \varepsilon (t)\| _{ L^2 } \le 1 \}.
\end{equation} 
Since $\varepsilon (T_n)= 0$, we have $- 1 \le \tau _n < T_n$, and it follows from~\eqref{fGCA14} that there exists a constant $C$ independent of $n$ such that
\begin{equation} \label{fGCA16} 
 \| \nabla \varepsilon \| _{ L^2 } \le C (T_n - t ) ^{ \frac { \widetilde{\rho} } {2} } 
\end{equation} 
for all $\tau _n \le t\le T_n$. 
This implies that there exists $\delta >0$ such that $\tau _n \le T_n -\delta $. 
This completes the proof.
\end{proof} 

\section{Proof of Theorem~$\ref{eThm1}$} \label{sCOMP} 

We consider the solution $u_n$ of equation~\eqref{NLS1} defined by~\eqref{fDFTN} and~\eqref{fGCA0}, $\varepsilon _n$ defined by\eqref{fGCA1}, and we set
\begin{gather*} 
V_n (t)= U( T_n -t ) \\
\eta _n (t)= \varepsilon _n ( T_n -t )
\end{gather*} 
for $t \ge 0$. 
It follows from~\eqref{feME1:1} that there exist $\delta , C >0$ such that
\begin{equation} \label{fEEN18} 
 \| \eta_n (t) \| _{ H^1 } \le C t ^{ \mu } ,\quad 0\le t\le \delta . 
\end{equation} 
Moreover, it follows from~\eqref{eEFP1} that
\begin{equation} \label{NLS5} 
\partial _t \eta _n = - i \Delta \eta_n - [g( V_n+ \eta_n ) - g(V_n) ] - i\Delta V_n 
\end{equation} 
Using the estimate $ | g(u+ v) - g(u) | \lesssim ( |u|^\alpha + |v|^\alpha ) |v| $ and the 
 embeddings $H^1 (\R^N ) \hookrightarrow L^{\alpha +2} (\R^N ) $, $L^{\frac {\alpha +2} {\alpha +1}} ( \R^N ) \hookrightarrow H^{-1} ( \R^N ) $, we deduce that
\begin{equation*} 
 \| \partial _t \eta _n \| _{ H^{-1} } \lesssim \| \eta_n \| _{ H^1 } + \| V_n \| _{ H^1 }^\alpha \| \eta_n \| _{ H^1 } + \| \eta_n \| _{ H^1 }^{\alpha +1} + \| \Delta V_n \| _{ L^2 }
\end{equation*} 
so that, applying~\eqref{fEEN18}, \eqref{fEU1:1:b1}, \eqref{fEU1:2:b1} and~\eqref{fEU1:3}, there exists $\kappa >0$ such that
\begin{equation} \label{fEEN19} 
 \| \partial _t \eta _n \| _{ H^{-1} } \le C t^{ - \kappa } ,\quad 0< t\le \delta .
\end{equation} 
Given $\tau \in (0, \delta )$, it follows from~\eqref{fEEN18} and~\eqref{fEEN19} that the sequence $(\eta _n) _{ n\ge 1 }$ is bounded in $L^\infty ((\tau , \delta ), H^1 (\R^N ) ) \cap W^{1, \infty } ((\tau ,\delta ), H^{-1} (\R^N ) )$. 
Therefore, after possibly extracting a subsequence, there exists $\eta \in L^\infty ((\tau ,\delta ), H^1 (\R^N ) ) \cap W^{1, \infty } ((\tau ,\delta ), H^{-1} (\R^N ) )$ such that
\begin{gather} 
\eta_n \goto _{ n\to \infty } \eta \text{ in } L^\infty ((\tau , \delta ), H^1 (\R^N ) ) \text{ weak$^\star$} 
\label{fEEN21} \\
\partial _t \eta_n \goto _{ n\to \infty } \partial _t \eta \text{ in } L^\infty ((\tau ,\delta ), H^{-1} (\R^N ) ) \text{ weak$^\star$} \label{fEEN22} \\
\eta_n (t) \goto _{ n\to \infty } \eta (t) \text{ weakly in $H^1 (\R^N ) $ and a.e. on $\R^N $, for all $\tau \le t\le \delta $} \label{fEEN23} 
\end{gather} 
Since $\tau \in (0, \delta )$ is arbitrary, a standard argument of diagonal extraction shows that there exists $\eta \in L^\infty _\Loc ((0 , \delta ), H^1 (\R^N ) ) \cap W^{1, \infty } _\Loc ((0 ,\delta ), H^{-1} (\R^N ) )$ such that (after extraction of a subsequence) \eqref{fEEN21}, \eqref{fEEN22} and~\eqref{fEEN23} hold for all $0<\tau <\delta $. 
Moreover, \eqref{fEEN18} and~\eqref{fEEN23} imply that
\begin{equation} \label{fEEN24} 
 \| \eta (t) \| _{ H^1 } \le C t ^{ \mu } ,\quad 0< t < \delta 
\end{equation} 
and \eqref{fEEN19} and~\eqref{fEEN22} imply that
\begin{equation} \label{fEEN24:b1} 
 \| \partial _t \eta \| _{ L^\infty ((\tau ,\delta ), H^{-1} )} \le C \tau ^{ -\kappa } 
\end{equation} 
for all $0 <\tau <\delta $. 
In addition, it follows easily from~\eqref{NLS5} and the convergence properties~\eqref{fEEN21}--\eqref{fEEN23} that 
\begin{equation} \label{NLS6} 
\partial _t \eta + i \Delta \eta = - [g( V + \eta ) - g(V) ] - i\Delta V
\end{equation} 
in $L^\infty _\Loc ((0 , \delta ), H^{-1} (\R^N ) )$, where $V(t) \equiv U( -t )$.
Therefore, setting 
\begin{equation} \label{fEEN25} 
 u(t)= U(t) + \eta (-t) \quad - \delta < t < 0 
\end{equation} 
we see that $u \in L^\infty _\Loc (( - \delta , 0 ), H^1 (\R^N ) ) \cap W^{1, \infty } _\Loc (( - \delta , 0 ), H^{-1} (\R^N ) )$ and that 
\begin{equation} \label{fEEN26:b1} 
iu_t + \Delta u= i g(u)
\end{equation} 
in $L^\infty _\Loc (( - \delta , 0 ), H^{-1} (\R^N ) )$. 
We now claim that 
\begin{equation} \label{fEEN27:b1} 
u\in C( (-\delta , 0), H^1 (\R^N ) ) \cap C^1 ( (-\delta , 0), H^{-1} (\R^N ) ) .
\end{equation} 
Indeed, let $0<\varepsilon < \delta $, and we consider the solution $u_1 \in C((-\infty , -\varepsilon ), H^1 (\R^N ) ) \cap C^1 ((-\infty , -\varepsilon ), H^{-1} (\R^N ) ) $ of~\eqref{NLS1} such that 
\begin{equation} \label{fUNQ1} 
u_1 (-\varepsilon ) = u (-\varepsilon ). 
\end{equation} 
(See Proposition~\ref{eRem3:1}.) 
The uniqueness property of Proposition~\ref{eRem3:1} implies that $u= u_1$ on $(-\delta , -\varepsilon )$.
In particular, $u \in C((-\delta , -\varepsilon ), H^1 (\R^N ) ) \cap C^1 ((-\delta , -\varepsilon ), H^{-1} (\R^N ) ) $, hence~\eqref{fEEN27:b1} follows, since $\varepsilon $ is arbitrary.
We may now extend $u$ for $t\le -\delta $ (see Proposition~\ref{eRem3:1}) to a solution $u\in C((-\infty , 0), H^1 (\R^N ) ) \cap C^1 ((-\infty , 0), H^{-1} (\R^N ) ) $ of~\eqref{NLS1}. 
Estimates~\eqref{fEEN26}, \eqref{fEEN27} and~\eqref{fEEN29} now follow from~\eqref{fEEN24}, \eqref{fEU1:1:b1} and~\eqref{fEU1:2:b1}, respectively, and the convergence property~\eqref{fEEN30} follows from~\eqref{fEEN26}.

\end{document}